\documentclass[a4paper]{amsproc}
\usepackage{amssymb}
\usepackage[hyphens]{url} 
\usepackage{amsmath}
\usepackage{cite}
\usepackage{mathrsfs}
\usepackage[title]{appendix} 
\usepackage[utf8]{inputenc}
\usepackage[english]{babel}
\usepackage{hyperref}
\usepackage{amsthm}
\usepackage{tikz-cd}
\usepackage{extarrows}
\usepackage{graphicx}
\usepackage{enumitem}
\usepackage{comment}
\hypersetup{linktoc=page}

\theoremstyle{plain}
\newtheorem{Theorem}{Theorem}

\newtheorem{Proposition}[Theorem]{Proposition}

\newtheorem{Lemma}[Theorem]{Lemma}
\newtheorem{Corollary}[Theorem]{Corollary}
\newtheorem{Example}[Theorem]{Example}

\theoremstyle{remark}
\newtheorem*{Remark}{Remark}

 \numberwithin{Theorem}{section}

\newcounter{desccount}
\newcommand{\descitem}[1]{%
  \item[#1] \refstepcounter{desccount}\label{#1}
}
\newcommand{\descref}[1]{\hyperref[#1]{#1}}

\title{Profinite non-rigidity of arithmetic groups}

\author[Amir Y. Weiss Behar]{\bfseries Amir Y. Weiss Behar}

\address{
Einstein Institute of Mathematics \\
Edmond J. Safra Campus (Givat-Ram)\\
The Hebrew University of Jerusalem\\
9190401\\
Israel}

\email{amir.behar@mail.huji.ac.il}

\begin{document}

\vspace{18mm} \setcounter{page}{1} \thispagestyle{empty}

\begin{abstract}
We show that for a typical high rank arithmetic lattice $\Gamma$, there exist finite index subgroups $\Gamma_{1}$ and $\Gamma_{2}$ such that $\Gamma_{1} \not\simeq \Gamma_{2}$ while $\widehat{\Gamma_{1}} \simeq \widehat{\Gamma_{2}}$. But there are exceptions to that rule.
\end{abstract}

\maketitle

\section{Introduction}

Let $\Gamma$ be a finitely generated residually finite group. We say that $\Gamma$ is profinitely-rigid if whenever $\widehat{\Lambda}\cong \widehat{\Gamma}$ for some finitely generated residually finite group $\Lambda$, then $\Lambda \cong \Gamma$. Here $\widehat{\Gamma}$ (resp. $\widehat{\Lambda}$) denotes the profinite completion of $\Gamma$ (resp. $\Lambda$). \newline
Up until recently, the only profinite rigid groups were "small" (i.e. without non-abelian free subgroups). Recently, in a groundbreaking work, Bridson, McReynolds, Reid and Spitler gave first examples of "big" groups which are profinitely rigid, among them are some fundamental groups of hyperbolic 3-manifolds \cite{BMRS20} and some triangle groups \cite{BMRS21}. \newline
Arithmetic subgroups of semisimple Lie groups need not be profinitely rigid (\cite{Ak12},\cite{KK20}). A well known open problem asks: 
\[
 \text{ For $n\geq 2$, is $SL_{n}(\mathbb{Z})$ profinitely rigid?} 
 \]
While we will not answer this question, we will show that there are finite index subgroups of these groups (at least when $n\geq 3$) which are not profinitely rigid. In fact, we will show a much more general result:
\begin{Theorem}[Main Theorem] \label{Main~The}
Let $k$ be a number field and $\mathbf{G}$ be a connected, simply connected, absolutely almost simple $k$-linear algebraic group of high $\infty$-rank such that $\mathbf{G}(k)$ satisfies the congruence subgroup property and $\Gamma \subseteq \mathbf{G}(k)$ an arithmetic subgroup.
\begin{enumerate}
\item Unless $\mathbf{G}$ is of type $G_2,F_4$ or $E_8$ and $k=\mathbb{Q}$, $\Gamma$ has infinitely many pairs of finite index subgroups $\Gamma_{1}$ and $\Gamma_2$ which are not isomorphic but their profinite completions are. 
\item The exceptional cases are truly exceptional and in these cases, there are no such pairs at all. In fact, if $\Gamma_{1},\Gamma_{2}\subseteq \mathbf{G}(\mathbb{Q})$ are arithmetic subgroups with isomorphic profinite completions then $\Gamma_{1}$ and $\Gamma_{2}$ are isomorphic.
\end{enumerate}
\end{Theorem}

Note that we show that each such $\Gamma$ has a finite index subgroup $\Gamma_{1}$ which is not profinitely rigid by showing that $\widehat{\Gamma_{1}}\cong \widehat{\Gamma_{2}}$ for some $\Gamma_{2}$ commensurable to it. This complements results of \cite{KK20},\cite{Ak12} and \cite{KS23} which give examples of non-commensurable arithmetic groups which are profinitely isomorphic. \newline
To illustrate our methods, let us now present them only for $\Gamma=SL_{4}(\mathbb{Z})$ (methods A \& B) and $\Gamma=SL_{2}(\mathbb{Z}[\sqrt{2}])$ (method C). \newline

\begin{description}
\descitem{Method A}  Using the centre of the simply connected form: Let $2 \neq p,q$ be two different primes, and let $\Lambda$ be the principle congruence subgroup corresponding to $pq$. Let $\rho_{p}$ be the element of $\widehat{SL_{4}(\mathbb{Z})}\cong \prod_{r}SL_{4}(\mathbb{Z}_{r})$ which is $1$ at the places $r\neq {p}$ and $-1$ at the place $p$, similarly define $\rho_{q}$. Define $\Gamma_{1} := \langle \widehat{\Lambda},\rho_{p}\rangle \cap SL_{4}(\mathbb{Z})$ and $\Gamma_{2} := \langle \widehat{\Lambda},\rho_{q}\rangle \cap SL_{4}(\mathbb{Z})$. Then $\widehat{\Gamma_{1}} \cong \widehat{\Lambda} \times \mathbb{Z}/2\mathbb{Z} \cong \widehat{\Gamma_{2}}$ but $\Gamma_{1}$ and $\Gamma_{2}$ cannot be isomorphic. For details see \ref{Proof~A}. \newline
\descitem{Method B} Using a non-trivial Dynkin automorphism: Let $2,3\neq p,q$ be two different primes. Consider the following maximal parabolic subgroups of $SL_{4}(\mathbb{Z}/p\mathbb{Z})$.
\[
P_{1,p}:=
\left\{ \begin{pmatrix} 
\ast & \ast & \ast & \ast \\
     & \ast & \ast & \ast \\
     & \ast & \ast & \ast \\
     & \ast & \ast & \ast \\
     \end{pmatrix} \right\}, 
\quad 
P_{2,p}:=
\left\{ \begin{pmatrix} 
\ast & \ast & \ast & \ast \\
\ast & \ast & \ast & \ast \\
\ast & \ast & \ast & \ast \\
     &      &      & \ast \\
     \end{pmatrix} \right\} 
\]
Similarly define the maximal parabolic subgroups $P_{1,q}$ and $P_{2,q}$ of $SL_{4}(\mathbb{Z}/q\mathbb{Z})$. Now, let $\Gamma_{1}$ be the congruence subgroup corresponding to $P_{1,p} \mod p$, $P_{1,q} \mod q$ and is trivial $\mod 3$, and  $\Gamma_{2}$ to be the congruence subgroup corresponding to $P_{1,p} \mod p$, $P_{2,q} \mod q$ and is trivial$\mod 3$. Then the profinite completions of $\Gamma_{1}$ and $\Gamma_{2}$ are isomorphic via the automorphism of $\widehat{SL_{4}(\mathbb{Z})}$ which is the non-trivial Dynkin automorphism at the place $q$ and the identity elsewhere, but they themselves cannot be isomorphic. For details see \ref{Proof~B}. \newline
\descitem{Method C} Using the number field: Let $2\neq p,q$ be two different primes such that $2$ is a square in $\mathbb{Q}_{p}$ and $\mathbb{Q}_{q}$, hence $p$ and $q$ splits completely in $\mathbb{Q}[\sqrt{2}]$. Set $\mathcal{O}=\mathbb{Z}[\sqrt{2}]$ and $\mathfrak{p}_{1},\mathfrak{p}_{2}|p$, $\mathfrak{q}_{1},\mathfrak{q}_{2}|q$ to be the primes lying over $p$ and $q$ respectively. Let $\Gamma_{1}$ be the principle congruence subgroup corresponding to $\mathfrak{p}_{1}\mathfrak{q}_{1}$, and $\Gamma_{2}$ be the principle congruence subgroup corresponding to $\mathfrak{p}_{2}\mathfrak{q}_{1}$. As $SL_{2}(\mathcal{O})$ has trivial congruence kernel, $\widehat{SL_{2}(\mathcal{O})} \cong \left( \prod_{l\neq p,q} SL_{2}(\mathcal{O}_{l}) \right) \times \prod_{i=1}^{2} \left( SL_{2}(\mathcal{O}_{\mathfrak{p}_{i}})\times SL_{2}(\mathcal{O}_{\mathfrak{q}_{i}}) \right)$, where for a prime $l\neq p,q$, $\mathcal{O}_{l}$ is the completion of $\mathcal{O}$ with respect to the primes lying over $l$. Then the profinite completions, $\widehat{\Gamma_{1}}$ and $\widehat{\Gamma_{2}}$ are isomorphic via the automorphism of $\widehat{SL_{2}(\mathbb{Z}[\sqrt{2}])}$ which is the transposition of the places $\mathfrak{p}_{1}$ and $\mathfrak{p}_{2}$. But $\Gamma_{1}$ and $\Gamma_{2}$ cannot be isomorphic. For details see \ref{Proof~C}. \newline
\end{description}

The main theorem is proved by generalizing the above methods to more general arithmetic lattices (In fact, only methods A and C are really needed). \newline
The paper is organized as follows: after some preliminaries in \S2, we will generalize methods A and C in \S3 deducing the first part of the main theorem. In \S4, we will elaborate on method B and finally in \S5 we will prove that the exceptional cases are true exceptions, concluding the main theorem. In \S6, we will give a stronger and more general version of the main theorem, stating that it holds for $S$-arithmetic groups and not merely for arithmetic groups. Moreover, one can get any (finite) number of non-isomorphic subgroup with isomorphic profinite completions (not just pairs). \newline

\textbf{Acknowledgments.} This work is a part of the author's PhD thesis at the Hebrew University.  For suggesting the above topic and for providing helpful guidance, suggestions and ideas I am deeply grateful to my advisors Alexander Lubotzky and Shahar Mozes. During the period of work on this paper I was supported by the European Research Council (ERC) under the European Union’s Horizon 2020 research and innovation programme (grant agreement No 882751) and by the ISF-Moked grant 2019/19.

\section{Preliminaries}
Throughout we assume that $k$ is a number field. The set of places of $k$ is denoted by $V(k)$, it is the union of the set of archimedean places $V_{\infty}(k)$ and the set of finite places $V_{f}(k)$. The completion of $k$ at $v\in V(k)$ is denoted by $k_{v}$. Let $\mathcal{O}_{k}$ denote the ring of integers of $k$ and for a finite place $v\in V_{f}(k)$, denote by $\mathcal{O}_{k,v}$ the ring of integers of $k_{v}$. The ring of finite adeles $\mathbb{A}_{k}^{f}= \prod^{*}_{v\in V_{f}(k)}k_{v} := \{ (x_{v})_{v}\in \prod_{v\in V_{f}(k)}k_{v}: \, x_{v}\in \mathcal{O}_{k,v} \text{ for all but finitely many places } \}$ is the restricted product over all the finite completions of $k$. If $k$ is clear from the context, we will omit the letter $k$ from all the above.\newline 
Let $\tilde{\mathbf{G}}$ be a connected, simply connected, absolutely almost simple $k$-linear algebraic group, with a fixed faithful $k$-representation $\rho:\tilde{\mathbf{G}} \to GL(n_{\rho})$. A subgroup $\Gamma \subseteq \tilde{\mathbf{G}}(k)$ is called arithmetic if it is commensurable with $\tilde{\mathbf{G}}(\mathcal{O})$ (see \cite{Mo01} and \cite{PR93}  for more details about arithmetic groups). We will usually write $\mathbf{G}$ for the adjoint form of $\tilde{\mathbf{G}}$ (which is the universal form), and by $\pi:\tilde{\mathbf{G}} \to \mathbf{G}$ the universal covering map, it is a central isogeny, and $\mathbf{C}:= \ker \pi = \mathcal{Z}(\tilde{\mathbf{G}})$ is a finite group. The $V_{\infty}(k)$-rank of $\tilde{\mathbf{G}}$ is $\text{rank}_{V_{\infty}(k)}\tilde{\mathbf{G}} := \sum_{v\in V_{\infty}(k)} \text{rank}_{k_{v}}\tilde{\mathbf{G}}$, where $\text{rank}_{k_{v}}\tilde{\mathbf{G}}$ is the dimension of a maximal $k_{v}$-split torus, $\tilde{\mathbf{G}}$ is said to have high $\infty$-rank if its $V_{\infty}(k)$-rank is $\geq 2$. \newline

We will use Margulis' superrigidity in a rather delicate manner. The particular version we use is the following:
\begin{Theorem}[Margulis' superrigidity]\label{Mar~Sup~Rig}
Assume $\tilde{\mathbf{G}}$ is of high $\infty$-rank, and let $\Gamma_{1},\Gamma_{2} \subseteq \tilde{\mathbf{G}}(k)$ be arithmetic subgroups. Assume further that $\Gamma_{i} \cap \mathbf{C}(k)={1}$. If $\varphi:\Gamma_{1}\to \Gamma_{2}$ is an isomorphism, then there exists a unique $k$-automorphism $\Phi$ of $\tilde{\mathbf{G}}$ and a unique automorphism $\sigma$ of $k$ such that $\varphi(\gamma)=\Phi(\sigma^{0}(\gamma))$ for every $\gamma\in \Gamma_{1}$, where $\sigma^{0}$ is the automorphism of $\tilde{\mathbf{G}}(k)$ induced by $\sigma$.
\end{Theorem}

\begin{proof}
Identifying $\Gamma_{1}$ and $\Gamma_{2}$ via the universal covering map $\pi:\tilde{\mathbf{G}}\to \mathbf{G}$ as arithmetic subgroups of the adjoint group $\mathbf{G}(k)$, Margulis' superrigidity \cite[Theorem VIII.3.6.(ii)]{Ma91} implies that the isomorphism $\varphi$ can be extended to an automorphism of $Res_{k/\mathbb{Q}}\mathbf{G}$. By the properties of the restriction of scalars functor, such an automorphism must be of the form $\Phi \circ \sigma^{0}$ for a $k$-automorphism $\Phi$ of $\mathbf{G}$ and an automorphism $\sigma$ of the field $k$ \cite[Proposition A.5.14]{CGP15}. Moreover, the $k$-automorphism $\Phi$ of the adjoint form $\mathbf{G}$ can be interpreted as a $k$-automorphism of the universal form $\tilde{\mathbf{G}}$, hence the assertion of the theorem.
\end{proof}

If $\sigma$ is an automorphism of $k$, it induces a permutation of the (finite) places of $k$, and thus an automorphism of the adelic group $\tilde{\mathbf{G}}(\mathbb{A}_{k}^{f})$ by permuting its factors according to $\sigma$, call this automorphism $\sigma_{\mathbb{A}}^{0}$. If $\Phi$ is a $k$-automorphism of $\tilde{\mathbf{G}}$, it induces a unique $k_{v}$-automorphism $\Phi_{v}:\tilde{\mathbf{G}}(k_{v})\to \tilde{\mathbf{G}}(k_{v})$ for every finite place $v$ of $k$ and the product $(\Phi_{v})_{v}:\prod_{v}\tilde{\mathbf{G}}(k_{v})\to \prod_{v}\tilde{\mathbf{G}}(k_{v})$ restricts to an automorphism $\Phi_{\mathbb{A}}:\tilde{\mathbf{G}}(\mathbb{A}_{k}^{f})\to \tilde{\mathbf{G}}(\mathbb{A}_{k}^{f})$ \cite[\S5]{PR93}.  Clearly $\Phi(\sigma^{0}(\gamma))=\Phi_{\mathbb{A}}(\sigma_{\mathbb{A}}^{0}(\gamma))$ for every $\gamma \in \tilde{\mathbf{G}}(k)$ (We identify the group of rational points $\tilde{\mathbf{G}}(k)$ with its diagonal embedding in the group of adelic points $\tilde{\mathbf{G}}(\mathbb{A}_{k}^{f})$) and is unique with this property. We thus get the following corollary:

\begin{Corollary}\label{cor~mar~sup}
    Under the assumptions of the previous theorem. If $\varphi:\Gamma_{1}\to \Gamma_{2}$ is an isomorphism. Then there exist unique automorphisms $\Phi_{\mathbb{A}}$ and $\sigma_{\mathbb{A}}^{0}$ of the adelic group $\tilde{\mathbf{G}}(\mathbb{A}_{k}^{f})$, such that $\sigma_{\mathbb{A}}^{0}$ is induced from an automorphism of $k$, and $\Phi_{\mathbb{A}}$ is induced from a $k$-automorphism of $\tilde{\mathbf{G}}$ with $\varphi(\gamma)=\Phi_{\mathbb{A}}(\sigma_{\mathbb{A}}^{0}(\gamma))$ for every $\gamma\in \Gamma_{1}$.
\end{Corollary}

We will also need an adelic version of Margulis' superrigidity stated and proven by Kammeyer and Kionke \cite[Theorem 3.2]{KK21}:
\begin{Theorem}\label{Ade~Sup~Rig}
Let $\mathbf{G}$ be a connected, absolutely almost simple $\mathbb{Q}$-linear algebraic group of high $\infty$-rank and $\Gamma\subseteq \mathbf{G}(\mathbb{Q})$ an arithmetic subgroup. If $\varphi:\Gamma \to \mathbf{G}(\mathbb{A}_{\mathbb{Q}}^{f})$ is a homomorphism such that $\overline{\varphi(\Gamma)}$ has non-empty interior, then there exist a homomorphism of adelic groups $\eta:\mathbf{G}(\mathbb{A}_{\mathbb{Q}}^{f}) \to \mathbf{G}(\mathbb{A}_{\mathbb{Q}}^{f})$, and a group homomorphism $\nu:\Gamma \to \mathcal{Z}(\mathbf{G})(\mathbb{A}_{\mathbb{Q}}^{f})$ with finite image such that $\varphi(\gamma)=\nu(\gamma)\eta(\gamma)$ for all $\gamma\in \Gamma$. Moreover, $\eta$ and $\nu$ are uniquely determined by this condition.
\end{Theorem}

\subsection{Profinite groups and the congruence subgroup property}\label{Pro~Grps}

A family $(\{G_{i}\}_{i\in I},\{\phi_{i,j}\}_{i\geq j \in I})$ is an inverse system of finite groups over the directed set $I$ if the $G_{i}$'s are finite groups, $\phi_{i,j}:G_{i}\to G_{j}$ are homomorphisms of groups whenever $i\geq j$ such that $\phi_{i,k}=\phi_{j,k}\phi_{i,j}$ for every $i\geq j \geq k$ and $\phi_{ii}=id_{G_{i}}$ $\forall i$. A group $G$ is called profinite if it is the inverse limit of an inverse system of finite groups over some directed set. A profinite group is a compact, Hausdorff, totally disconnected topological group, a map of profinite groups is a continuous homomorphism of groups.

\begin{Example}[Profinite completion]
    Let $\Gamma$ be a finitely generated group, let $\mathcal{N}$ be the set of finite index normal subgroups of $\Gamma$, for $M,N\in \mathcal{N}$, declare that $M\leq N$ whenever $N\subseteq M$, it is a directed set. Consider the natural quotient homomorphisms $\phi_{N,M}:\Gamma/N \to \Gamma/M$, then the profinite group $\widehat{\Gamma}=\varprojlim_{N\in \mathcal{N}} \Gamma/N$, is called the \textbf{profinite completion} of $\Gamma$.
\end{Example}
The profinite completion $\widehat{\Gamma}$ and the set $\mathcal{C}(\Gamma)$ of isomorphism classes of the finite quotients of $\Gamma$ hold the same information in following manner:

\begin{Theorem}\cite[Theorems 3.2.2 \& 3.2.7]{RZ00} 
    If $\Gamma$ and $\Lambda$ are two finitely generated residually finite groups then $\mathcal{C}(\Gamma)=\mathcal{C}(\Lambda)$ if and only if $\widehat{\Gamma} \cong \widehat{\Lambda}$
\end{Theorem}

There is a natural map $\iota:\Gamma \to \widehat{\Gamma}$ given by $\gamma\mapsto (\gamma N)_{N}$, this map is injective if and only if $\Gamma$ is residually finite, in this case we identify $\Gamma$ with its image $\iota(\Gamma)$. The pair $(\widehat{\Gamma},\iota)$ satisfies a universal property: $\iota(\Gamma)$ is dense in $\widehat{\Gamma}$, and for every profinite group $P$, and every homomorphism $\varphi:\Gamma\to P$, there exists a unique homomorphism of profinite groups $\hat{\varphi}:\widehat{\Gamma} \to P$ such that $\widehat{\phi} \circ \iota = \varphi$.

There is a strong connection between the finite index subgroups of $\widehat{\Gamma}$ and those of $\Gamma$:

\begin{Proposition}\cite[Proposition 3.2.2]{RZ00}\label{Pro~Corre}
    Let $\Gamma$ be a finitely generated residually finite group, then there is a one-to-one correspondence between the set $\mathcal{X}$ of all finite index subgroups of $\Gamma$ and the set $\mathcal{Y}$ of all finite index subgroup of $\widehat{\Gamma}$, given by
    \begin{gather*}
        X \mapsto \overline{X}, \quad X \in \mathcal{X} \\
        Y \mapsto Y\cap \Gamma, \quad Y\in \mathcal{Y}
    \end{gather*}
    where $\overline{X}$ denote the closure of $X$ in $\widehat{\Gamma}$. Moreover, this bijection preserves normality, index and quotients. 
\end{Proposition}

\begin{Example}[Congruence completion]
    Let $\Gamma$ be an arithmetic subgroup of $\tilde{\mathbf{G}}(k)$. Consider the set $\mathcal{C}$ of all congruence subgroups, i.e. subgroups that contain  $\Gamma[\mathcal{I}] := \Gamma \cap \ker(\phi_{\mathcal{I}}:\tilde{\mathbf{G}}(\mathcal{O}_{k})\to \tilde{\mathbf{G}}(\mathcal{O}_{k}/\mathcal{I}))$ for some ideal $\mathcal{I}\triangleleft \mathcal{O}_{k}$, where $\phi_{\mathcal{I}}$ is the reduction map $\mod \mathcal{I}$. As in the profinite completion, $\mathcal{C}$ is a directed set by the inverse of inclusion, and one can form the \textbf{Congruence completion} $\overline{\Gamma}$ of $\Gamma$ with respect to this inverse system. 
\end{Example}

Thus, there is a surjective map $\widehat{\Gamma} \to \overline{\Gamma}$ between the profinite completion and the congruence completion. Call $C(\Gamma)$, the kernel of this map, the \textbf{congruence kernel}. The group $\Gamma$ is said to have the \textbf{congruence subgroup property} if the congruence kernel $C(\Gamma)$ is a finite group. It is not difficult to see that the congruence subgroup property is actually a property of the ambient group $\tilde{\mathbf{G}}$ and the field $k$. \newline
It was conjectured by Serre \cite{Se70} that if $\text{rank}_{V_{\infty}(k)}(\tilde{\mathbf{G}})\geq 2$ and $\Gamma\subseteq \tilde{\mathbf{G}}(k)$ is an arithmetic subgroup then $C(\Gamma)$ is trivial or isomorphic to a subgroup of the roots of unity of $k$. The conjecture has been proven in many instances, including for example, all the isotropic cases \cite{Ra86} and all anisotropic groups of type $B_{n}$,$C_{n}$,$D_{n}$ (except for some triality forms of $D_{4}$), $E_{7}$,$E_{8}$,$F_{4}$ and $G_{2}$ (\cite[Ch.9]{PR93}, \cite{PR10}).

\subsection{A number theoretic lemma}

\begin{Lemma} \label{num~the~lem}
    Let $\tilde{\mathbf{G}}$ be a connected, simply connected, absolutely almost simple $k$-linear algebraic group. There exist infinitely many finite places $v\in V_f(k)$ such that $\tilde{\mathbf{G}}$ splits over $k_{v}$. Moreover, one can assume that for these places, $\mathbf{C}(\mathcal{O}_{v})=\mathbf{C}(\mathbb{C})$.
\end{Lemma}

\begin{proof}
    There exists a finite Galois field extension $k'/k$ such that $\tilde{\mathbf{G}}$ splits over $k'$ and $\mathbf{C}(k')=\mathbf{C}(\mathbb{C})$. By Chebotarev's density theorem \cite[Corollary 13.6]{Ne13}, there exist infinitely many primes $\mathfrak{p}\subseteq \mathcal{O}_{k}$ that splits completely in $k'$. In particular, if such a prime lies under a prime $\mathfrak{p'}\subseteq \mathcal{O}_{k'}$ then $k_{v}\cong k'_{v'}$ where $v$ and $v'$ are the places corresponding to the primes $\mathfrak{p}$ and $\mathfrak{p'}$ respectively. Thus, for such a place $v$, $\tilde{\mathbf{G}}$ splits over $k_{v}$. \newline
    Moreover, as the centre $\mathbf{C}(k')$ is finite, for all but finitely many places, $\mathbf{C}(k')=\mathbf{C}({\mathcal{O}_{k',v'}})$. So there exist infinitely many places $v\in V_{f}(k)$ with $\mathbf{C}(\mathcal{O}_{k,v})=\mathbf{C}(\mathbb{C})$ and $\tilde{\mathbf{G}}$ splits over $k_{v}$.
\end{proof}

\section{ First part of the main theorem - Existence}

In this section we will prove the first part of theorem \ref{Main~The}, the existence part, it will follow from the two theorems below. 

\begin{Theorem}\label{Proof~A}
    Let $\Gamma \subseteq \tilde{\mathbf{G}}(k)$ be an arithmetic subgroup. Assume that $\tilde{\mathbf{G}}$ has the congruence subgroup property and type different then $E_{8},F_{4}$ or $G_{2}$. Then there exist two non-isomorphic finite index subgroups $\Gamma_{1},\Gamma_{2}\subseteq \Gamma$ with isomorphic profinite completions.
\end{Theorem}

\begin{proof}
    (Following \descref{Method A}) Moving to a finite index subgroup, one can assume that there exists a finite set of primes $S$ and a compact open subgroup $\Lambda \subseteq \prod_{v\in S} \tilde{\mathbf{G}}(k_{v})$ commensurable with $\prod_{v\in S}\tilde{\mathbf{G}}(\mathcal{O}_{v})$ such that 
    \begin{gather}
        \widehat{\Gamma} = \overline{\Gamma} \cong \Lambda \times \prod_{v\notin S} \tilde{\mathbf{G}}(\mathcal{O}_{v}).
    \end{gather}
    Indeed, the congruence kernel $C(\Gamma) \subseteq \widehat{\Gamma}$ is finite, so one can find a finite index subgroup $\Gamma'^{P}\subseteq \widehat{\Gamma}$ of the above form, by proposition \ref{Pro~Corre}, there exists a finite index subgroup $\Gamma'\subseteq \Gamma$ such that $\widehat{\Gamma'} = \Gamma'^{P}$. Let $\pi:\tilde{\mathbf{G}} \to \mathbf{G}$ be the canonical central isogeny to the adjoint form $\mathbf{G}$, by moving again to a finite index subgroup one can assume that $\Gamma \cap \mathbf{C}(k) = \{1\}$, and $\widehat{\Gamma}$ is still of the same form as $(1)$. \newline
    Let $S_{Q}$ be the set of all rational primes lying under some valuation in $S$, and $S_{\text{full}}=\{ v : \, v|p \text{ for some } p \in S_{Q} \}$. 
    By lemma \ref{num~the~lem} one can find two valuations $\mathfrak{p},\mathfrak{q}\notin S_{\text{full}}$ lying over different rational primes $p$ and $q$ respectively and such that $\mathbf{C}(\mathcal{O}_{\mathfrak{p}})=\mathbf{C}(\mathcal{O}_{\mathfrak{q}})=\mathbf{C}(\mathbb{C})$, and thus also elements $1\neq \rho_{\mathfrak{p}}\in \mathbf{C}(\mathcal{O}_{\mathfrak{p}})$ and $1\neq \rho_{\mathfrak{q}}\in \mathbf{C}(\mathcal{O}_{\mathfrak{q}})$ of the same order. For every $v|p,q$, let $\Delta_{v}$ be a finite index subgroup of $\tilde{\mathbf{G}}(\mathcal{O}_{v})$ with $\mathbf{C}(\mathcal{O}_{v}) \cap \Delta_{v} = \{1\}$.
    Define $\epsilon_{\mathfrak{p}} = (\epsilon_{\mathfrak{p},w})_{w},\epsilon_{q} = (\epsilon_{\mathfrak{q},w})_{w} \in \widehat{\Gamma}$ where
    \[
    \epsilon_{\mathfrak{p},w} := \begin{cases} 1 & \quad \text{if } w \neq \mathfrak{p} \\ \rho_{\mathfrak{p}} & \quad \text{if } w=\mathfrak{p} \end{cases}, \quad 
    \epsilon_{\mathfrak{q},w} := \begin{cases} 1 & \quad \text{if } w \neq \mathfrak{q} \\ \rho_{\mathfrak{q}} & \quad \text{if } w=\mathfrak{q} \end{cases}
    \]
   Now we define the following subgroups of $\widehat{\Gamma}$:
    \begin{gather*}
        \Delta = \Lambda \times \prod_{v \notin S, \, v \nmid p, \, v\nmid q} \tilde{\mathbf{G}}(\mathcal{O}_{v}) \times \prod_{v|p \text{ or } v|q} \Delta_{v} \\
        \Delta_{1} = \langle \Delta,\epsilon_{\mathfrak{p}} \rangle \cong \Delta \times \langle \epsilon_{\mathfrak{p}} \rangle \\
        \Delta_{2} = \langle \Delta, \epsilon_{\mathfrak{q}} \rangle \cong \Delta \times \langle \epsilon_{\mathfrak{q}} \rangle 
    \end{gather*}
    Clearly $\Delta_{1}$ and $\Delta_{2}$ are isomorphic finite index subgroups of $\widehat{\Gamma}$. By proposition \ref{Pro~Corre}, there exist finite index subgroups $\Gamma_{i} \subseteq \Gamma$ with $\widehat{\Gamma_{i}} = \Delta_{i}$ for $i=1,2$. We will finish the proof by showing that $\Gamma_{1}$ and $\Gamma_{2}$ cannot be isomorphic. \newline
     Assume to the contrary that there exists an isomorphism $\varphi:\Gamma_{1} \to \Gamma_{2}$. By corollary \ref{cor~mar~sup} there exist unique adelic automorphisms $\Phi_{\mathbb{A}}$ and $\sigma_{\mathbb{A}}^{0}$ of $\tilde{\mathbf{G}}(\mathbb{A}_{k}^{f})$ such that $\sigma_{\mathbb{A}}^{0}$ is induced from an automorphism of $k$ and $\Phi_{\mathbb{A}}$ is induced from a $k$-automorphism of $\tilde{\mathbf{G}}$ such that $\varphi(\gamma)=\Phi_{\mathbb{A}}(\sigma_{\mathbb{A}}^{0}(\gamma))$ for every $\gamma\in \Gamma_{1}$, taking closures one deduces that $(\tilde{\varphi}_{\mathbb{A}}\circ \sigma_{\mathbb{A}})(\widehat{\Gamma_{1}}) = \widehat{\Gamma_{2}}$. On the other hand, the induced map between the profinite completions $\widehat{\varphi}:\widehat{\Gamma_{1}} \to \widehat{\Gamma_{2}}$ is unique with $\widehat{\varphi}(\gamma)=\varphi(\gamma)$ for every $\gamma \in \Gamma_{1}$. Thus, it must be that $\widehat{\varphi} \equiv (\tilde{\varphi}_{\mathbb{A}}\circ \sigma_{\mathbb{A}})|_{\widehat{\Gamma_{1}}}$. In particular, it implies that the $\mathfrak{p}$'th place of $\widehat{\Gamma_{1}}$ is mapped isomorphically onto the $\sigma(\mathfrak{p})$'th place of $\widehat{\Gamma_{2}}$. This is a contradiction since $\sigma(\mathfrak{p})|p$, so the $\sigma(\mathfrak{p})$'th place of $\widehat{\Gamma_{2}}$ is centerless, but the $\mathfrak{p}$'th place of $\widehat{\Gamma_{1}}$ has a non-trivial centre.
\end{proof}

\begin{Theorem}\label{Proof~C}
Let $\Gamma \subseteq \tilde{\mathbf{G}}(k)$ be an arithmetic subgroup. Assume further that $\tilde{\mathbf{G}}(k)$ has the congruence subgroup property and that $k$ is a number field of dimension $d\geq 2$ over $\mathbb{Q}$. Then there exist two non-isomorphic finite index subgroups $\Gamma_{1},\Gamma_{2}\subseteq \Gamma$ with isomorphic profinite completions.
\end{Theorem}

\begin{proof}
(Following \descref{Method C}) As before, by moving to a finite index subgroup we can assume that $\Gamma \cap \mathbf{C}(k) = \{1\}$ and that $\widehat{\Gamma} = \overline{\Gamma} = \Lambda \times \prod_{v\notin S} \tilde{\mathbf{G}}(\mathcal{O}_{v})$, for some finite set of places $S$ and $\Lambda \subseteq \prod_{v\in S}\tilde{\mathbf{G}}(k_{v})$ commensurable with $\prod_{v\in S}\tilde{\mathbf{G}}(\mathcal{O}_{v})$. Let $S_{Q}$ and $S_{\text{full}}$ be as before. By Chebotarev's density theorem there exist infinitely many rational primes that split completely in $k$ \cite[Corollary 13.6]{Ne13}, pick two such different primes $p,q\notin S_{Q}$. Say $\mathfrak{p}_{1},...,\mathfrak{p}_{d}|p$ and $\mathfrak{q}_{1},...,\mathfrak{q}_{d}|q$. For $i=1,2$, let $\Gamma_{i}:=\Gamma(\mathfrak{p}_{i}\mathfrak{q}_{1})$ be the principle congruence subgroups modulo $\mathfrak{p}_{i}\mathfrak{q}_{1}$, then obviously $\widehat{\Gamma_{1}}\cong \widehat{\Gamma_{2}}$. \newline
 By corollary \ref{cor~mar~sup}, if $\Gamma_{1}$ and $\Gamma_{2}$ are isomorphic, then there exist unique adelic automorphisms $\Phi_{\mathbb{A}}$ and $\sigma_{\mathbb{A}}^{0}$ of $\tilde{\mathbf{G}}(\mathbb{A}_{k}^{f})$ such that $\sigma_{\mathbb{A}}^{0}$ is induced from an automorphism of $k$ and $\Phi_{\mathbb{A}}$ is induced from a $k$-automorphism of $\tilde{\mathbf{G}}$ such that $\varphi(\gamma)=\Phi_{\mathbb{A}}(\sigma_{\mathbb{A}}^{0}(\gamma))$ for every $\gamma\in \Gamma_{1}$, taking closures one deduces that $(\tilde{\varphi}_{\mathbb{A}}\circ \sigma_{\mathbb{A}})(\widehat{\Gamma_{1}}) = \widehat{\Gamma_{2}}$. As $\sigma_{\mathbb{A}}$ acts by permuting the places and $\tilde{\varphi}_{\mathbb{A}}$ acts place-wise it must be that $\sigma_{\mathbb{A}}:\mathfrak{p}_{1} \mapsto \mathfrak{p}_{1}$. We claim that such $\sigma$ must be trivial, which will finish the proof. \newline
Indeed, let $N$ be the Galois closure of $k$, $G:=Gal(N/\mathbb{Q}),\, H:=Gal(N/k)$, let $\mathfrak{r}$ be a prime lying over $\mathfrak{p}_{1}$ and $G_{\mathfrak{r}}$ be its decomposition group. The correspondence $H \backslash G / G_{\mathfrak{r}} \to \{ \mathfrak{p}_{1},....,\mathfrak{p}_{d} \}$ given by $H\sigma G_{\mathfrak{r}} \mapsto \sigma \mathfrak{p}_{1}$ is a one to one correspondence (of $G$-sets), hence $G_{\mathfrak{r}}$ must be trivial. Thus, $Aut(k)$ acts freely on the set of primes $\{ \mathfrak{p}_{1},...,\mathfrak{p}_{d} \}$, as needed.
\end{proof}

\section{ Another set of examples }

The following theorem is not needed for the proof of the main theorem, but following method B, it gives many more examples of non-profinitely rigid arithmetic groups.

\begin{Theorem}\label{Proof~B}
        Let $\Gamma \subseteq \tilde{\mathbf{G}}(k)$ be an arithmetic subgroup. Assume further that $\tilde{\mathbf{G}}$ has: \textbf{1)} the congruence subgroup property; \textbf{2)} type $A_{n},D_{n}$ or $E_{6}$; and \textbf{3)} there exists an archimedean place $k_{v}$ such that $\text{rank}_{k_{v}}\tilde{\mathbf{G}} \geq 2$. Then there exist two non-isomorphic finite index subgroups subgroups $\Gamma_{1},\Gamma_{2}\subseteq \Gamma$ with isomorphic profinite completions.
\end{Theorem}
Note that $A_{n},D_{n},E_{6}$ are exactly the types of Dynkin diagrams with a non-trivial symmetry.

\begin{proof}
    (Following \descref{Method B}) As before, by moving to a finite index subgroup we can assume that $\Gamma\cap \mathbf{C}(k) = \{1\}$ and that $\widehat{\Gamma} = \overline{\Gamma} = \Lambda \times \prod_{v\notin S} \tilde{\mathbf{G}}(\mathcal{O}_{v})$, for some finite set of places $S$ and $\Lambda \subseteq \prod_{v\in S} \tilde{\mathbf{G}}(k_{v})$ commensurable with $\prod_{v\in S} \tilde{\mathbf{G}}(\mathcal{O}_{v})$. \newline
    By lemma \ref{num~the~lem}, one can find two different primes $\mathfrak{p},\mathfrak{q} \notin S_{\text{full}}$ (where $S_{\text{full}}\supseteq S$ is as in the previous section) lying over different rational primes, such that $\tilde{\mathbf{G}}$ splits over both $k_{\mathfrak{p}}$ and $k_{\mathfrak{q}}$. Fix root systems for $\tilde{\mathbf{G}}(k_{\mathfrak{p}})$ and $\tilde{\mathbf{G}}(k_{\mathfrak{q}})$, and let $\Sigma$ be the set of simple roots, identified for both groups. For $v=\mathfrak{p},\mathfrak{q}$, the reduction maps $\pi_{v}:\tilde{\mathbf{G}}(\mathcal{O}_{v})\to \tilde{\mathbf{G}}(\mathcal{O}_{v}/v)$ are onto, and $G_{v}:=\tilde{\mathbf{G}}(\mathcal{O}_{v}/v)$ is the split universal Chevalley group of the same type as $\tilde{\mathbf{G}}$ over the residue field. The root systems we fixed define corresponding root systems for $G_{v}$.\newline
    For each of the types involved, there exist a non-trivial symmetry $s$ of the Dynkin diagram, pick a subset $\theta \subseteq \Sigma$ which is non-invariant under $s$, e.g.
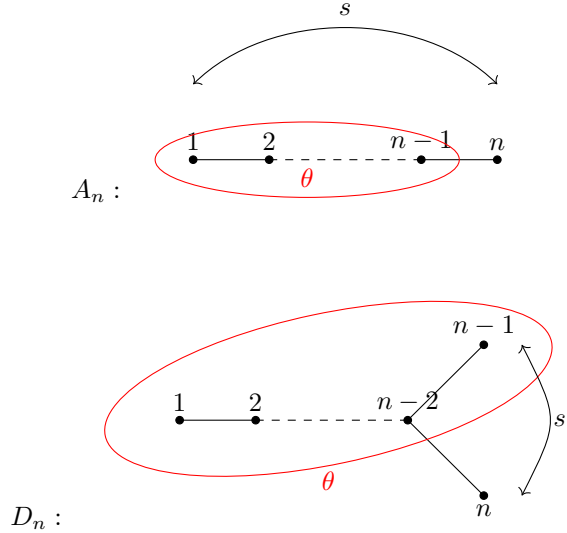
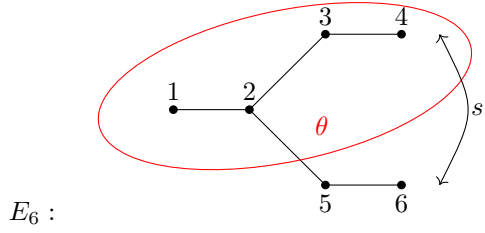
    \begin{gather*}
    A_{n} : \quad
    \begin{tikzpicture} 
        \draw (0,0) -- (1,0);
        \filldraw (0,0) circle (0.05 cm) node[anchor=south] {$1$};
        \filldraw (1,0) circle (0.05 cm) node[anchor=south] {$2$};
        \draw[dashed] (1,0) -- (3,0);
        \draw (3,0) -- (4,0);
        \filldraw (3,0) circle (0.05 cm) node[anchor=south] {$n-1$};
        \filldraw (4,0) circle (0.05 cm) node[anchor=south] {$n$};
        \draw[red] (1.5,0) ellipse (2cm and 0.5 cm) node[anchor=north] {$\theta$};
        \draw[<->] (0,1) .. controls (1,2) and (3,2) ..(4,1);
        \draw (2,2) node {$s$};
    \end{tikzpicture}
    \end{gather*}
    
    \begin{gather*}
    D_{n}: \quad
        \begin{tikzpicture} 
        \draw (0,0) -- (1,0);
        \filldraw (0,0) circle (0.05 cm) node[anchor=south] {$1$};
        \filldraw (1,0) circle (0.05 cm) node[anchor=south] {$2$};
        \draw[dashed] (1,0) -- (3,0);
        \draw (3,0) -- (4,1);
        \draw (3,0) -- (4,-1);
        \filldraw (3,0) circle (0.05 cm) node[anchor=south] {$n-2$};
        \filldraw (4,1) circle (0.05 cm) node[anchor=south] {$n-1$};
        \filldraw (4,-1) circle (0.05 cm) node[anchor=north] {$n$};
        \draw[red, rotate=12] (2,0) ellipse (3cm and 1 cm) node[anchor=north, outer sep = 28pt] {$\theta$};
        \draw[<->] (4.5,1) .. controls (5,0) .. (4.5, -1);
        \draw (5,0) node {$s$};
    \end{tikzpicture}
    \end{gather*}

    \begin{gather*}
        E_{6}: \quad
        \begin{tikzpicture} 
        \draw (0,0) -- (1,0);
        \filldraw (0,0) circle (0.05 cm) node[anchor=south] {$1$};
        \filldraw (1,0) circle (0.05 cm) node[anchor=south] {$2$};
        \draw (1,0) -- (2,1);
        \draw (1,0) -- (2,-1);
        \draw (2,1) -- (3,1);
        \draw (2,-1) -- (3,-1);
        \filldraw (2,1) circle (0.05 cm) node[anchor=south] {$3$};
        \filldraw (3,1) circle (0.05 cm) node[anchor=south] {$4$};
        \filldraw (2,-1) circle (0.05 cm) node[anchor=north] {$5$};
        \filldraw (3,-1) circle (0.05 cm) node[anchor=north] {$6$};
        \draw[red, rotate=12] (1.5,0) ellipse (2.5cm and 1 cm) node[anchor=north west, outer sep = 8pt] {$\theta$};
        \draw[<->] (3.5,1) .. controls (4,0) .. (3.5, -1);
        \draw (4,0) node {$s$};
        \end{tikzpicture}
    \end{gather*}

The symmetry $s$ of the Dynkin diagram induces an isomorphism $\varphi_{v}$ of $G_{v}$ \cite[Corollary to theorem 29]{St67}. For a subset $R\subseteq \Sigma$ of simple roots, let $P_{v,R}$ be the parabolic subgroup of $G_{v}$ corresponding to $R$, then $P_{v,\theta}$ and $P_{v,s \theta}$ are non-conjugate in $G_{v}$, but isomorphic via $\varphi_{v}$.  \newline
Now, let $\pi:\Gamma \to G_{\mathfrak{p}} \times G_{\mathfrak{q}}$ be the canonical reduction map modulo $\mathfrak{p} \mathfrak{q}$. Consider the following two congruence subgroups
\begin{gather*}
    \Gamma_{1} := \pi^{-1}(P_{\mathfrak{p},\theta} \times P_{\mathfrak{q}, \theta}); \\
    \Gamma_{2} := \pi^{-1}(P_{\mathfrak{p}, s\theta} \times P_{\mathfrak{q}, \theta}).
\end{gather*}
Then $\widehat{\Gamma_{1}}$ and $\widehat{\Gamma_{2}}$ are isomorphic via $\Phi:= (\Phi_{v})_{v}$ where $\Phi_{v} \equiv id$ for $v \neq \mathfrak{p}$ and $\Phi_{\mathfrak{p}}$ is the isomorphism of $\tilde{\mathbf{G}}(k_{\mathfrak{p}})$ induced by the non-trivial symmetry $s$ of the Dynkin diagram. It remains to show that $\Gamma_{1}$ and $\Gamma_{2}$ cannot be isomorphic. \newline
Assume to the contrary that there exists an isomorphism $\varphi:\Gamma_{1} \to \Gamma_{2}$. By corollary \ref{cor~mar~sup}, there exist unique adelic automorphisms $\Phi_{\mathbb{A}}$ and $\sigma_{\mathbb{A}}^{0}$ of $\tilde{\mathbf{G}}(\mathbb{A}_{k}^{f})$ such that $\sigma_{\mathbb{A}}^{0}$ is induced from an automorphism of $k$ and $\Phi_{\mathbb{A}}$ is induced from a $k$-automorphism of $\tilde{\mathbf{G}}$ such that $\varphi(\gamma)=\Phi_{\mathbb{A}}(\sigma_{\mathbb{A}}^{0}(\gamma))$ for every $\gamma\in \Gamma_{1}$. Moreover $\Phi_{\mathbb{A}}$ is of the form $C_{g}\cdot \omega_{\mathbb{A}}$ where $C_{g}$ is conjugation by some $g\in \tilde{\mathbf{G}}(\overline{k})$, $\omega$ is an outer automorphism which comes from a symmetry of the Dynkin diagram and $\omega_{\mathbb{A}}$ just acts as $\omega$ at each place. In particular $\pi_{\mathfrak{p}}^{-1}(P_{\mathfrak{p},\theta})$ is mapped onto $g\pi_{\mathfrak{p}}^{-1}(P_{\sigma(\mathfrak{p}),\omega\theta})g^{-1}$ and $\pi_{\mathfrak{q}}^{-1}(P_{\mathfrak{q},\theta})$ is mapped onto $g\pi_{\mathfrak{q}}^{-1}(P_{\sigma(\mathfrak{q}),\omega\theta})g^{-1}$. By our choice of $\mathfrak{p}$ and $\mathfrak{q}$ it must be that $\sigma(\mathfrak{p})=\mathfrak{p}$ and $\sigma(\mathfrak{q})=\mathfrak{q}$.\newline
We have that $g^{-1}\pi_{\mathfrak{p}}^{-1}(P_{\mathfrak{p},\theta})g=\pi_{\mathfrak{p}}^{-1}(P_{\mathfrak{p},\omega\theta})$, multiplying $g$ from both sides by elements of $\tilde{\mathbf{G}}(k)$, we can assume that conjugation by $g$ preserves the root system that was fixed in the beginning of the proof. If $\theta \neq \omega\theta$, let $\alpha$ be a simple root in $\omega\theta \backslash \theta$, then the action of $g$ on the one-parameter unipotent subgroup $U_{\alpha}$ must be as scalar multiplication by some $f_{\alpha}$ with $\text{val}_{\mathfrak{p}}(f_{\alpha})=1$. Thus, the action on the opposite one-parameter unipotent subgroup $U_{-\alpha}$ is given as scalar multiplication by $1/f_{\alpha}$, but then $g^{-1}\pi_{\mathfrak{p}}^{-1}(P_{\mathfrak{p},\theta})g\nsubseteq \pi_{\mathfrak{p}}^{-1}(P_{\mathfrak{p},\omega\theta})$. Hence $\theta$ must be equal to $\omega\theta$. The same argument imply that $\omega\theta=s\theta$, and so $\theta = \omega\theta = s\theta$. But $\theta$ was chosen to be non-invariant under $s$, a contradiction.
\end{proof}

\section{ Second part of the main theorem - The exceptional cases}

\begin{Lemma}
    Let $\mathbf{G}$ be a connected, simply connected and absolutely almost simple $\mathbb{Q}$-linear algebraic group of type $E_{8},F_{4}$ or $G_{2}$. Then $\mathbf{G}$ splits over $\mathbb{Q}_{p}$ for every $p$.
\end{Lemma}

\begin{proof}
    The $\mathbb{Q}_{p}$ forms of the group $\mathbf{G}$ are classified by the first Galois cohomology set $H^{1}(\mathbb{Q}_{p},Aut(\mathbf{G}))$. For the groups considered, the universal and the adjoint forms coincide, moreover there are no symmetries for their Dynkin diagram, hence $Aut(\mathbf{G}) \cong \mathbf{G}$. As the field $\mathbb{Q}_{p}$ is local and non-archimedean, by \cite[Theorem 6.4]{PR93}, the Galois cohomology group $H^{1}(\mathbb{Q}_{p},\mathbf{G})$ is trivial. Thus, there is only one (up to an isomorphism) $\mathbb{Q}_{p}$ form for $\mathbf{G}$, in particular, this form must be the split form.
\end{proof}

\begin{Remark}\label{aut~spl~grp}
If $\mathbf{G}$ is a $k$-split simple $k$-linear algebraic group, then the automorphism group, $Aut(\mathbf{G}(k))$, of $\mathbf{G}(k)$ is completely known. Precisely, following the notations of \cite[Theorem 30]{St67}, each automorphism can be written as the product of an inner, a diagonal, a graph and a field automorphism. We have used graph automorphisms for \descref{Method B} and field automorphisms for \descref{Method C}. The group of diagonal automorphisms (modulo the inner ones) has a connection with the centre of the universal form \cite[Exercise following theorem 30]{St67}, which was used for \descref{Method A}. Thus, restriction ourselves to the exceptional cases, where $\mathbf{G}$ has type $E_{8},F_{4}$ or $G_{2}$ and $k=\mathbb{Q}$ (and also for $\mathbb{Q}_{p}$), all automorphisms are inner.
\end{Remark}

\begin{Theorem}
    Let $\mathbf{G}$ be a connected, simply connected and absolutely almost simple high $\infty$-rank $\mathbb{Q}$-linear algebraic group of type $E_{8},F_{4}$ or $G_{2}$. If $\Gamma_{1},\Gamma_{2}\subseteq \mathbf{G}(\mathbb{Q})$ are two arithmetic subgroups with isomorphic profinite completions, then $\Gamma_{1}$ and $\Gamma_{2}$ are isomorphic.
\end{Theorem}

\begin{proof}
Let $\Gamma_1,\Gamma_2 \subseteq \mathbf{G}(\mathbb{Q})$ be two arithmetic subgroups. As noted in the preliminaries, the congruence kernel is trivial for these groups, so one can write $\overline{\Gamma_{i}} = \widehat{\Gamma_{i}} = \Lambda_{i} \times \prod_{p\notin S}\mathbf{G}(\mathbb{Z}_{p})$ where $S$ is a finite set and $\Lambda_{i}$ are commensurable with $\prod_{p\in S}\mathbf{G}(\mathbb{Z}_{p})$.  Assume that $\Phi:\overline{\Gamma_{1}} \cong \widehat{\Gamma_{1}} \xrightarrow{\sim} \widehat{\Gamma_{2}} \cong \overline{\Gamma_{2}} \subseteq \mathbf{G}(\mathbb{A})$ is an isomorphism between the profinite completions of the two.
By adelic supperrigidity \ref{Ade~Sup~Rig}, there exists a unique homomorphism of adelic groups
\begin{gather*}
        \tilde{\Phi}:\mathbf{G}(\mathbb{A}^{f}) \to \mathbf{G}(\mathbb{A}^{f})  
\end{gather*}
such that $\tilde{\Phi}|_{\Gamma_{1}} \equiv \Phi\circ \iota|_{\Gamma_{1}}$, as $\mathbf{G}(\mathbb{A}^{f})$ is centerless by our assumption on the type of $\mathbf{G}$. Moreover, using the uniqueness of the map, $\tilde{\Phi}$ must be an isomorphism, and $\tilde{\Phi}|_{\overline{\Gamma_{1}}} \equiv \Phi$. \newline
Consider the homomorphisms $\tilde{\Phi}_{p,q}:\mathbf{G}(\mathbb{Q}_{p}) \to \mathbf{G}(\mathbb{Q}_{q})$ which are the composite
\begin{gather*}
    \mathbf{G}(\mathbb{Q}_{p}) \xrightarrow{\iota_{p}} \mathbf{G}(\mathbb{A}^{f}) \xrightarrow{\tilde{\Phi}} \mathbf{G}(\mathbb{A}^{f}) \xrightarrow{\pi_{q}} \mathbf{G}(\mathbb{Q}_{q})
\end{gather*}
of the inclusion in the $p$'th place, $\tilde{\Phi}$, and the projection onto the $q$'th place. This is a continuous homomorphism between a $p$-adic group and a $q$-adic group, so if $p\neq q$ it must be a locally constant map. So its image is a normal countable subgroup of $\mathbf{G}(\mathbb{Q}_{q})$, in particular it is not of finite index, and hence must be trivial \cite[Proposition 3.17]{PR93}. Thus, $\tilde{\Phi} = (\tilde{\Phi}_{p,p})_{p}$ is given by an isomorphism at each place, and $\tilde{\Phi}_{p,p}$ must be conjugation by some $y'_{p}\in \mathbf{G}(\mathbb{Q}_{p})$ (see the remark above). \newline
We truncate $\tilde{\Phi}$ in the following manner, 
\[
\text{write }y= (y_{p})_{p},\text{ where }y_{p} = \begin{cases} y_{p} = y'_{p} & \quad p\in S \\ y_{p}=1 & \quad p\notin S \end{cases}.
\]
By our choice of $S$,  conjugation by $y$ is again an isomorphism between $\overline{\Gamma_{1}}$ and $\overline{\Gamma_{2}}$. By the Strong Approximation theorem \cite[Theorem 7.12]{PR93}, there exists some $g\in \mathbf{G}(\mathbb{Q})$ with $g\in y\overline{\Gamma_{1}}$. Thus, conjugation by $g\in \mathbf{G}(\mathbb{Q})$ is an isomorphism between $\overline{\Gamma_{1}}$ and $\overline{\Gamma_{2}}$. We have that $\Gamma_{i} = \overline{\Gamma_{i}}\cap \mathbf{G}(\mathbb{Q})$, which imply that $\Gamma_{1}$ and $\Gamma_{2}$ can be conjugated by $g$, as needed.
\end{proof}

\section{Final Remarks}

It is possible to generalize our methods even further. For example, using \descref{Method A}, we can find finite index subgroups of $\Gamma:= SL_{2}(\mathbb{Z}[1/p])$ which are not profinitely rigid. Explicitly (for $p\neq 2,3,5$), the following finite index subgroups of $SL_{2}(\mathbb{Z}[1/p])$ are non-isomorphic, but their profinite completions are:
\begin{gather*}
    \Gamma_{1} := \left\{ \begin{pmatrix} a & b \\ c & d \end{pmatrix}\in SL_{2}(\mathbb{Z}[1/p]) :\, \begin{matrix} b,c \equiv 0 \mod 3,5 \\ a,d \equiv \pm 1 \mod 3 \\ a,d \equiv 1 \mod 5 \end{matrix}   \right\}; \\
    \Gamma_{2} := \left\{ \begin{pmatrix} a & b \\ c & d \end{pmatrix}\in SL_{2}(\mathbb{Z}[1/p]) :\, \begin{matrix} b,c \equiv 0 \mod 3,5 \\ a,d \equiv  1 \mod 3 \\ a,d \equiv \pm 1 \mod 5 \end{matrix}   \right\}.
\end{gather*}
Indeed, just as in \S3, $\widehat{\Gamma_{1}}\cong \widehat{\Gamma(15)}\times \mathbb{Z}/2\mathbb{Z} \cong \widehat{\Gamma_{2}}$, where $\Gamma(15)$ is the principle congruence subgroup of $\Gamma$ of level 15. \newline 
Let us stress out that it is still unknown whether or not $SL_{2}(\mathbb{Z}[1/p])$ itself is profinitely rigid, and in fact, there are some reasons to believe it is profinitely rigid (see for example \cite[\S 4]{CTLR22}). On the other hand, increasing slightly the dimension, it has been shown that $SL_{4}(\mathbb{Z}[1/p])$ is not profinitely rigid \cite{CTLR22}. \newline
We would like to state a stronger version of Theorem \ref{Main~The} which includes the above example. First, we need some further notations. Let $S\subseteq V(k)$ be a finite set of places containing all the archimedean places. The ring of $S$-integers of the number field $k$ is 
\[
\mathcal{O}_{k,S}:=\{ x\in k:\, v(x)\geq 0 \, \forall v\notin S\}
\]
Let $\tilde{\mathbf{G}}$ be a connected, simply connected, absolutely almost simple $k$-linear algebraic group with a fixed faithful $k$-representation $\rho:\tilde{\mathbf{G}}\to GL(n_{\rho})$ . A subgroup $\Gamma \subseteq \tilde{\mathbf{G}}(k)$ is called an $S$-arithmetic subgroup if it is commensurable with $\tilde{\mathbf{G}}(\mathcal{O}_{S})$. As in \S2, there is a map from the profinite completion to the congruence completion, denote its kernel by $C(\Gamma,S)$. The group $\Gamma$ is said to have the congruence subgroup property (with respect to $S$) if $C(\Gamma,S)$ is a finite group. Again, this is actually a property of the ambient group $\tilde{\mathbf{G}}$, the field $k$ and the set $S$. The proofs given throughout the paper, carry over to establish:
\begin{Theorem}
    Let $n$ be a positive integer, $k$ a number field, $S$ a finite set of places of $k$ containing all the archimedean places, $\tilde{\mathbf{G}}$ a connected, simply connected and absolutely almost simple $k$-linear algebraic group such that $\sum_{v\in S} \text{rank}_{k_{v}}\tilde{\mathbf{G}} \geq 2$ and such that $\tilde{\mathbf{G}}(k)$ satisfies the congruence subgroup property (with respect to $S$). Let $\Gamma \subseteq \tilde{\mathbf{G}}(k)$ be an $S$-arithmetic subgroup. Then, unless $\tilde{\mathbf{G}}$ has type $G_{2},F_{4}$ or $E_{8}$ and $k=\mathbb{Q}$, $\Gamma$ has infinitely many sequences of pairwise non-isomorphic finite index subgroups $\Gamma_{1},...,\Gamma_{n}$ with isomorphic profinite completions.
\end{Theorem}

As in the main theorem, the exceptional cases are indeed exceptional. Moreover, in these cases, if $\Gamma_{1}$ is an $S_{1}$-arithmetic subgroup and $\Gamma_{2}$ is an $S_{2}$-arithmetic subgroup with $\widehat{\Gamma_{1}}\cong \widehat{\Gamma_{2}}$ then $S_{1}=S_{2}$ and $\Gamma_{1}\cong\Gamma_{2}$.

\bibliographystyle{acm}
\bibliography{References}

\nocite{*}

\end{document}